\documentclass[reqno, 12pt]{amsart}
\usepackage[letterpaper,hmargin=1in,vmargin=1in]{geometry}
\usepackage{amsmath,url}
\usepackage{amssymb}
\usepackage{amsthm}
\usepackage{soul}
\usepackage[percent]{overpic}
\usepackage{graphicx}
\usepackage[utf8]{inputenc}
\usepackage{cite}
\usepackage[hidelinks]{hyperref}
\newcommand {\Real}{\mathbb{R}}
\newcommand{\mco}{\mathcal{O}}
\newtheorem{rem}{Remark}

\numberwithin{equation}{section}
 
\usepackage[usenames,dvipsnames]{color}

\usepackage{pict2e}

\newtheorem{theorem}{Theorem}

\title[Unconditional wave decay in dimension two] {Unconditional wave decay in dimension two}
\author{T. J. Christiansen, K. Datchev, P. Morales, and M.  Yang}
\address{Department of Mathematics, University of Missouri, Columbia, MO 65211 USA}
\email{christiansent@missouri.edu}
\address{Department of Mathematics, Purdue University, West Lafayette, IN 47907 USA}
\email{kdatchev@purdue.edu}
\address{Department of Mathematics, Purdue University, West Lafayette, IN 47907 USA}
\email{moralep@purdue.edu}
\address{Department of Mathematics, University of California, Berkeley, CA 97420 USA}
\email{mxyang@math.berkeley.edu}
\begin{document}
\begin{abstract}We extend Burq's logarithmic  wave decay rate \cite{burq98} to general compactly supported scatterers in dimension two. The main novelty is using recent results on low-frequency expansions to remove the requirement that the spectrum be regular at zero. This allows us to include, among other examples, arbitrary smooth obstacles with variable boundary conditions.
\end{abstract}
\maketitle

\section{Introduction}

Accurate estimation and computation of decay rates is an important problem in the study of linear and nonlinear waves \cite{dz,sch21,hh,bs}. In \cite{burq98}, Burq showed that, for Dirichlet obstacles in $\mathbb R^d$, waves decay at least logarithmically fast, \textit{regardless} of the shape of the obstacle. This optimal universal estimate has since been generalized very widely (see \cite{mos,shapiro24,gh}  for some recent results and references), but always requiring suitable regularity of the spectrum at low frequencies.

Such spectral regularity holds naturally for many examples, such as the Dirichlet or Neumann Laplacian on an exterior domain. But not always: in general, lower order terms and variable boundary conditions contribute nontrivially in dimension two. In this paper, we use recent results on wave expansions \cite{cd2,cdy} to remove the regularity requirement, and prove an unconditional result. We focus on $d=2$ because this dimension has the richest low frequency phenomena, but Theorem \ref{t:main} below is an abstract result applicable in any dimension, and also in non-Euclidean (e.g.~conic and Aharonov--Bohm) scattering, as in \cite{cdy}.

Our motivating example is the classical obstacle problem. Let $\mathcal B \subset \mathbb R^2$ be a ball, and let $\mathcal O \subset \mathcal B$ be an \textit{obstacle}: i.e. a (possibly empty) open set with smooth boundary, such that $\Omega = \mathbb R^2 \setminus \overline{\mathcal O}$ is connected. Let $P$ be an elliptic, self-adjoint, differential operator on $\Omega$ which equals the Laplacian away from $\mathcal B$. In this paper, we compute asymptotics for the wave equation associated to $P$.

More precisely, let
\begin{equation}
    \label{eq:operator}
     P = -\sum_{j,k=1}^2 a_{j,k}(x) \partial_{x_j}\partial_{x_k} + \sum_{j=1}^2 b_j(x) \partial_{x_j} + V(x),
\end{equation}
where $a_{j,k} - \delta_{j,k}$, $b_j$, and $V$ are all in $C_c^\infty(\mathbb R^2)$, and  the matrix $a_{j,k}(x)$ is positive definite for all $x$. We write the boundary $\partial \Omega$ as
\[\partial \Omega = \Gamma_D \cup \Gamma_N,\]
where $\Gamma_D$ and $\Gamma_N$ are separated (and either or both may be empty). Let $\gamma  \in C^\infty(\Gamma_N)$, and equip $P$ with the domain
\begin{equation}
    \label{eq:domain}
     \mathcal D = \{ u \in H^2(\Omega) \colon u|_{\Gamma_D} = 0, \ (\partial_\nu u + \gamma u)|_{\Gamma_N} = 0\},
\end{equation}
where $\partial_\nu$ is the normal derivative to the boundary with respect to the metric $g = (a_{j,k})^{-1}$. 

\begin{theorem}\label{t:intro}
Let $P$ be as in \eqref{eq:operator}, and suppose that $P$ is self-adjoint  on $L^2(\Omega)$ with domain $\mathcal D$. Let $f \in C_c^\infty(\Omega)$, and let $u \in C^\infty(\Omega \times \mathbb R)$ be the solution to 
\[
(\partial_t^2 + P)u(x,t) = 0, \qquad u|_{t=0} = 0, \quad \partial_t u|_{t=0} =f, \quad u(\cdot,t) \in \mathcal D.
\]
Then
\[
 u(x,t) = u_d(x,t) + u_z(x,t) + u_r(x,t).
\]
Here
\begin{enumerate}
\item[(i)] $u_d$ is the contribution of any negative eigenvalues: 
\[u_d(x,t) = \sum_{\ell=1}^{N_E} \frac{\sinh(\sqrt{-E_\ell}t)}{\sqrt{-E_\ell}}(\Pi_\ell f)(x),\] where $E_1 < \cdots < E_{N_E} <0$ are the negative eigenvalues of $P$, and $\Pi,\dots,\Pi_{N_E}$ the corresponding $L^2(\Omega)$ projections.
\item[(ii)] $u_z$ is the contribution of zero energy:
\begin{equation}
    \label{eq:uz}
     u_z(x,t) = t (\Pi_0f)(x) + \sum_{m=1}^M \langle f, U_{\omega_m}\rangle\mathfrak{J}_m(t) U_{\omega_m}(x),
\end{equation}
where $\Pi_0\colon L^2(\Omega) \to L^2(\Omega)$ the projection onto the zero eigenspace of $P$, $M \in \{0,1,2\}$,  $\mathfrak{J}_m$ is defined in \eqref{eq:Jm} and satisfies $\mathfrak{J}_m(t) = t((\log t)^{-1} + O((\log t)^{-2}))$ as $t \to \infty$, and the $U_{\omega_m}$ are certain $p$-resonant states. 
\item[(iii)] $u_{r}$ is a remainder, tending to zero logarithmically fast, in the sense that for every compact set $K \subset \Omega$, positive number $M$ and multiindex $\alpha$, we have
\begin{equation}\label{e:ur}
\lim_{t \to \infty} \log(t)^M\max_{x \in K} |\partial^\alpha u_r(x,t)| = 0.
\end{equation}
\end{enumerate}
\end{theorem}
The remainder bound \eqref{e:ur} follows from a more precise estimate bounding each Sobolev norm of the remainder $u_r$ in terms of a Sobolev norm of the initial data $f$; see Theorem~\ref{t:main} below. It is proved using Burq's universal exponential resolvent bound \cite{burq98, burq02}. This exponential bound is known to hold for broader classes of operators, and in particular we may allow $P$ to be an operator on a manifold, and to have rougher coefficients; see \cite{mps20,llr,dgs, o24, shapiro24,gras,gh,v25} for some recent results and further references in this active area.

It is well-known that a remainder bound of the form \eqref{e:ur} is optimal at our level of generality. This follows from an argument of Ralston \cite{r69}, deducing wave non-decay from sequences of quasimodes; see Theorem 1.2 of \cite{dmst} for a recent result of this kind in a setting close to the present one, and \cite{hs,keir,gh} for recent such results in general relativity.

The most interesting terms in our expansions are those in \eqref{eq:uz} coming from $p$-resonances at zero, when $P$ has these. Recall that $u\in \mathcal{D}_{\operatorname{loc}}$ is a $p$-resonant state if $Pu = 0$ and $u \in L^q(\Omega) \Leftrightarrow q>2$, i.e. $u$ just barely fails to be in $L^2(\Omega)$. Such resonances can occur when we allow variable lower order terms in the coefficients of the operator or the boundary condition, as in the case of the Robin (impedance) boundary condition. See Section \ref{sec:p-res} for examples.

With more control of high frequencies, we can continue the expansion. If the resolvent obeys a polynomial bound, as in situations of no trapping or mild trapping, the next terms are additional zero energy contributions: see \cite{cdy}. On the other hand, if there is stable trapping, then resonances approach the real axis quickly, and they should contribute next; see Theorem 7.20 of \cite{dz} for a related result.

Prior work on this problem has obtained the result of Theorem \ref{t:intro} under assumptions that require regularity at zero energy which rules out terms $u_d$ and $u_z$, for example requiring $P$ to be a Laplace--Beltrami operator with Dirichlet or Neumann (not Robin) boundary conditions \cite{burq98}, or with a spectral cutoff away from zero energy \cite{burq02}; see \cite{mos,shapiro24,gh}  for some recent results and references.

Wave decay results similar to ours have been much studied for decades. The field is too wide-ranging to survey here. Let us mention the seminal work of Morawetz \cite{m61}, and the surveys and works in \cite[Epilogue]{lp89}, \cite[Chapter X]{va}, \cite{dr}, \cite{tat}, \cite{dz}, \cite{vasy}, \cite{sch21}, \cite{hh}, \cite{klainerman}, \cite{lo24}, \cite{lsv}.

\subsection*{Structure of the paper}
In Section \ref{sec:thm2}, we prove our main abstract theorem on the asymptotic expansion of solutions to the wave equation in a general setting. In Section \ref{sec:thm1}, we apply Theorem \ref{t:main} to prove Theorem \ref{t:intro} on the wave asymptotics for the operator $P$ in \eqref{eq:operator} with domain $\mathcal{D}$ in \eqref{eq:domain}. In Section \ref{sec:p-res}, we give examples of operators having $p$-resonances.  

\section{The Main Expansion}
\label{sec:thm2}

Let $P$ be  self-adjoint operator on a Hilbert space $\mathcal{H}$ with domain $\mathcal{D}$. We denote the resolvent of $P$ by $R(\lambda)=(P- \lambda^2)^{-1}$, mapping $\mathcal{H} \rightarrow \mathcal{H}$ whenever $\text{Im}\, \lambda>0$ and $\lambda^2$ is not an eigenvalue of $P$. For $s\geq 0$, define $\mathcal{D}^s = \langle P\rangle^{-s}\mathcal{H}$, with norm \[\|g \|_{\mathcal{D}^s}=\|\langle P \rangle^s \,g \|_{\mathcal{H}}, \qquad  \langle P\rangle =(P^2+I)^{1/2}.\]

We make the following assumptions:

\begin{enumerate}
    \item The spectrum of   $P$ consists of $[0,\infty)$ together with up to finitely many negative eigenvalues $E_1 < \cdots < E_n < 0$ with corresponding orthogonal projections $\Pi_1,\, \ldots \,, \Pi_n$.
    \item There exist positive constants $\lambda_0$, $C$, $C'$, and a bounded operator $\chi\colon \mathcal H \to \mathcal H$, such that $\chi R(\lambda)\chi$ continues analytically to
    \[\left\{\lambda \in \mathbb C \colon  |\text{Im}\, \lambda| < \frac{1}{C} e^{-C'|\text{Re}\,\lambda|},\, \text{Re}\, \lambda \neq 0\ \right\}. \] 
      Moreover, for some $p,q\geq0$, we have
    \[ \|\chi R(\lambda)\chi\|_{\mathcal{D}^p \rightarrow \mathcal{D}^q} \leq C|\lambda|^{-2}e^{C'|\text{Re}\, \lambda|}, \qquad \text{when }  |\text{Im}\, \lambda| < \frac{1}{C} e^{-C'|\text{Re}\,\lambda|}, \ \text{Re} \lambda \ne 0.\]
    \item There are operators $\chi A_{j,k}\chi:\mathcal{D}^p \to \mathcal{D}^q$, with $p,q$ as in Assumption (2), and numbers $\nu_j \in \mathbb{R}$, $b_{j,k} \in \mathbb{C} \backslash i[0,\infty)$, such that
    \[
    \chi R(\lambda)\chi = \sum_{j,k} \chi A_{j,k} \chi \lambda^{\nu_j} \log^k(b_{j,k}\lambda),
    \]
    where the sum $\sum_{j,k}$ converges absolutely and uniformly on compact subsets of 
    \[
    \left\{ \lambda \mid \arg \lambda \in \left( -\frac{\pi}{2},\frac{3 \pi}{2}\right), |\lambda| < \lambda_0 \right\}.
    \]    
 \end{enumerate}

 We consider the solution $u(t) = \frac{\sin(t\sqrt{P})}{\sqrt{P}} f$ of the wave equation  
\[
\begin{cases} 
(\partial_t^2 +P)u =0, \\
u|_{t=0}=0,\quad \partial_t u|_{t=0} = f.
\end{cases}
\]

  To analyze the solution, use the following contours. 
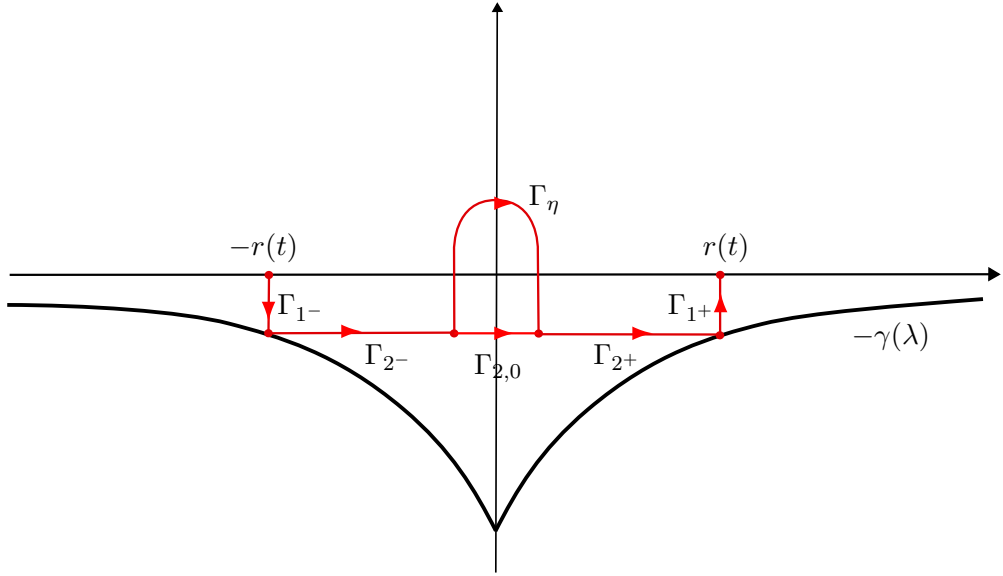
\begin{figure}[h]
  \centering
  \begin{overpic}[width=0.8\textwidth]{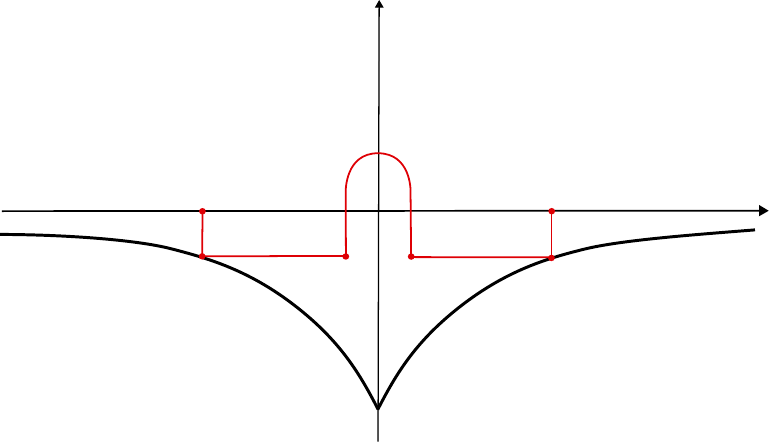}
    \put(66.5, 26){\small\text{$\Gamma_{1^+}$}}
    \put(27.2, 26){\small\text{$\Gamma_{1^-}$}}
    \put(59, 21){\small\text{$\Gamma_{2^+}$}}
    \put(36, 21){\small\text{$\Gamma_{2^-}$}}
    \put(52.5,37){\small \text{$\Gamma_\eta$}}
    \put(85, 23){\small\text{$-\gamma(\lambda)$}}
    \put(22.3, 32){\small\text{$-r(t)$}}
    \put(70, 32){\small\text{$r(t)$}}
    \put(45,24.1){\color{red}\linethickness{0.3mm}\polygon(0,0)(8.5,0)}
    \put(26.3,26.2){\color{red}\linethickness{0.4mm}\vector(0,-1){1}}
    \put(71.8,27.3){\color{red}\linethickness{0.4mm}\vector(0,1){1}}
    \put(34.7,24.3){\color{red}\linethickness{0.4mm}\vector(1,0){1}}
    \put(64,24.1){\color{red}\linethickness{0.4mm}\vector(1,0){1}}
    \put(50.1,24.1){\color{red}\linethickness{0.4mm}\vector(1,0){1}}
    \put(50.1,37.2){\color{red}\linethickness{0.4mm}\vector(1,0){1}}
    \put(47,20){\small \text{$\Gamma_{2,0}$}}
  \end{overpic}
  \caption{The contours of integration.}
  \label{fig:contour1}
\end{figure}

\noindent Here, $\Gamma = \Gamma _{1^\pm} \cup \Gamma_{2^\pm} \cup \Gamma_\eta$ where
\begin{align*}
    &\Gamma_{1^\pm} = \{ \pm r(t) +iy \mid - \gamma(r(t)) \leq y \leq 0\}, \\
    &\Gamma_{2^\pm} = \{\pm x -i \gamma(r(t)) \mid \eta < x< r(t) \},\\
    &\Gamma_{2,0} = \{\pm x -i \gamma(r(t)) \mid -\eta < x< \eta \},\\
    &r(t)=\frac{\log(t)}{A}, \qquad \gamma(\lambda) =  \frac 1 C e^{-C'|\text{Re}\, \lambda|},
\end{align*}
$C$ and $C'$ are as in Assumption (2), $A>C'$, 
and $\Gamma_\eta$ consists of the semicircle from $-\eta$ to $i\eta$ to $\eta$, together with two vertical line segments:
\begin{figure}[h]
  \centering
  \begin{overpic}[width=0.25\textwidth]{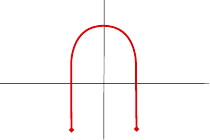} 
    \put(69, 2){\small\text{$\eta - i \gamma(r(t))$}}
    \put(-24, 2){\small\text{$-\eta - i \gamma(r(t))$}}  
    \put(55,54){\color{red}\linethickness{0.4mm}\vector(1,0){1}}
  \end{overpic}
  \caption{The contour $\Gamma_\eta$.}
  \label{fig:contour2}
\end{figure} 

\begin{theorem}\label{t:main}
    Fix $\chi$ as above such that $\chi:\mathcal{D}^q \to \mathcal{D}^q$ is bounded, and $s\geq 0$. If $\chi f = f \in \mathcal{D}^{p+s}$, then \[\chi u(t) = \chi (u_d(t) +u_z(t) +u_r(t)),\] where
    \begin{enumerate}
        \item $u_d$ is the contribution of the point spectrum: 
        \begin{equation}
            \label{eq:ud}
            u_d(t) = \sum_{\ell=1}^n  \frac{\sinh (\sqrt{-E_{\ell}} t )}{\sqrt{-E_{\ell}}}\Pi_{\ell} f.
        \end{equation}
        \item $u_z$ is the zero energy contribution: \[\chi u_z(t) = \frac{1}{2\pi} \sum_{j,k} \chi A_{j,k} \chi f \lim_{\eta \rightarrow 0^+} \int_{\Gamma_\eta} e^{-it \lambda} \lambda^{\nu_j} \log^k (b_{j,k}\lambda) d\lambda. \]
        \item $u_r$ is a logarithmic remainder: \[\|\chi u_r(t) \|_{\mathcal{D}^q} = O\left(\frac{1}{\log(t)^{2(s+p-q)+1}}\right) \|f\|_{\mathcal{D}^{p+s}}. \]
    \end{enumerate}
    The convergence in the zero energy contribution and the bound on the remainder are uniform for  $t\in [2,\infty)$.
\end{theorem}

\begin{proof}
    Fix $f\in \mathcal{D}^{p+s}$ with $\chi f=f$ and take $g \in \mathcal{D}^p$ such that $f=\langle P \rangle^{-s}g$. 
    Using $\lambda^2$ as the spectral parameter of the operator $P$, and considering the point spectrum of $P$, by functional calculus,
    \begin{align}\label{eq:lgsm}
        \frac{\sin(t\sqrt{P})} {\sqrt{P}}f
        &= \frac{\sin(t\sqrt{P})} {\sqrt{P}}\langle P\rangle^{-s}g \nonumber \\
        &= \mathbf{1}_{[0,r(t)^2]} (P) \frac{\sin(t\sqrt{P})}{\sqrt{P}}\,\langle P\rangle^{-s}g + \mathbf{1}_{(r(t)^2,\infty)} (P) \frac{\sin(t\sqrt{P})} {\sqrt{P}}\,\langle P\rangle^{-s}g + u_d(t)
    \end{align} 
with $u_d(t)$ given by \eqref{eq:ud}. We begin by estimating the second term. Again using the functional calculus,
    \begin{align*}
        &\left \|\langle P \rangle^{-q} \langle P \rangle^q \left (\mathbf{1}_{(r(t)^2,\infty)} (P) \frac{\sin(t\sqrt{P})} {\sqrt{P}}\,\langle P\rangle^{-s}\right) \langle P \rangle^{-p} \langle P \rangle^p\right \|_{\mathcal{D}^{p} \to \mathcal{D}^q} \\
        &\leq \left\| \langle P \rangle^q \left (\mathbf{1}_{(r(t)^2,\infty)} (P) \frac{\sin(t\sqrt{P})} {\sqrt{P}}\,\langle P\rangle^{-s}\right) \langle P \rangle^{-p}\right \|_{\mathcal{H} \to \mathcal{H}}\\
        &\leq \sup_{r(t) < \lambda < +\infty} \left |\langle \lambda^2\rangle^q \left(\frac{\sin\, (t \lambda)}{\lambda} \right) \langle \lambda^2\rangle ^{-s}\langle \lambda^2\rangle^{-p} \right|\\
        &\leq \sup_{r(t) < \lambda < +\infty} \left | \frac{\sin\, (t \lambda)}{\lambda^{2(s+p-q)+1}} \right | \\
        &\leq \frac{1}{r(t)^{2(s+p-q)+1}}
    \end{align*}
    
    \noindent where in the first inequality, we used the fact that $\| \langle P\rangle^{-q} \|_{\mathcal{H} \to \mathcal{D}^q}=1$ and $\|\langle P\rangle^p\|_{\mathcal{D}^p \to \mathcal{H}}=1.$ 
    Thus we deduce \[
    \left\| \mathbf{1}_{(r(t)^2,\infty)} (P) \frac{\sin(t\sqrt{P})} {\sqrt{P}}\,\langle P\rangle^{-s}g \right\|_{\mathcal{D}^q}  = O\left(\frac{1}{\log(t)^{2(s+p-q)+1}}\right)  \|f\|_{\mathcal{D}^{p+s}}.  
    \]

     Now we turn to the first term of \eqref{eq:lgsm}. Using Stone's formula and a change of variables \[\mathbf{1}_{[0,r(t)^2]} (P) \frac{\sin(t\sqrt{P})}{\sqrt{P}}\,\langle P\rangle^{-s}g =\frac{1}{2\pi} \int_{-r(t)}^{r(t)} e^{-it \lambda} (R(\lambda +i0) - R(-\lambda +i0))\,\langle \lambda^2\rangle^{-s}g \,d\lambda. \]
    
    \noindent Multiplying on the left by $\chi$
    and using its mapping properties and that
    $\chi f=f$
    \[\chi \left(\mathbf{1}_{[0,r(t)^2]} (P) \frac{\sin(t\sqrt{P})}{\sqrt{P}}\,\langle P\rangle^{-s}g\right)=\frac{1}{2\pi} \int_{-r(t)}^{r(t)} e^{-it \lambda} \chi(R(\lambda +i0) - R(-\lambda +i0))\chi \,\langle \lambda^2\rangle^{-s}g \,d\lambda. \]

    \noindent To have analyticity of the function $\langle \lambda^2 \rangle$ we use the branch cut  $e^{\pm i\frac{\pi}{4}}[1,\infty)\cup e^{\pm i\frac{3\pi}{4}}[1,\infty)$ and work on the strip $\left\{\lambda \mid \text{Re}\,\lambda \in \mathbb{R}, -\frac{{1}}{2} < \text{Im}\,\lambda < \frac{1}{2}\right\}.$  \\

    \noindent We estimate the contribution of $R(-\lambda +i0)$ by shifting the integral downward to the contour $\Gamma_{1^-}\cup \Gamma_{2^-}\cup \Gamma_{2,0} \cup \Gamma_{2^+} \cup \Gamma_{1^+}$ of Figure \ref{fig:contour1}; for this term there is no problem with the negative imaginary axis, as the integrand is analytic in the open rectangle from $\lambda = (-r(t),-\gamma(r(t))$ to $(r(t),0)$. This yields
    \begin{align*} 
    &\left\| \int_{-r(t)}^{r(t)} e^{-it \lambda} \chi R(-\lambda+i0) \chi \langle \lambda^2 \rangle ^{-s} g\, d\lambda \right\|_{\mathcal{D}^q}  \\ 
    &\leq  \left\| \int_{-r(t)}^{r(t)} e^{-it(\lambda - i\gamma(r(t))} \chi R(-\lambda +i \gamma(r(t)) \chi\, \langle (\lambda-i\gamma(r(t)))^2 \rangle ^{-s} g d\lambda \right\|_{\mathcal{D}^q} + \left\| \int_{\Gamma_1\pm} \right\|_{\mathcal{D}^q}
    \end{align*}

    \noindent Considering just the first term in the right hand side and letting the constant $C$  change from line to line, we get
    \begin{align*}
        &\left\| \int_{-r(t)}^{r(t)} e^{-it(\lambda - i\gamma(r(t))} \chi R(-\lambda +i \gamma(r(t)) \chi\, \langle (\lambda-i\gamma(r(t)))^2 \rangle ^{-s} g d\lambda \right \|_{\mathcal{D}^q} \\
        &\leq e^{-t \gamma(r(t))} \int_{-r(t)}^{r(t)} \|\chi R(-\lambda + i\gamma (r(t))) \chi \|_{\mathcal{D}^p \to \mathcal{D}^q} \, \left|\langle (\lambda-i\gamma(r(t)))^2\rangle\right|^{-s}\, \|g \|_{D^p} \,d \lambda\\
        &\leq e^{-t \gamma(r(t))} \int_{-r(t)}^{r(t)} Ce^{C'| \text{Re}\, (-\lambda +i\gamma (r(t)))|} \,  d\lambda \, \|g\|_{D^p} \\
        &\leq e^{-t \gamma(r(t))}\, C e^{C'r(t)} \, r(t) \,\|f\|_{\mathcal{D}^{p+s}}.
    \end{align*}
    
    \noindent where the second inequality uses the fact that $|\langle \lambda^2 \rangle| \geq \frac{1}{2}$ which can be obtained by writing $|1+\lambda^4| \geq |\text{Re}(1+\lambda^4)|$, completing squares with respect to $\text{Re}\, \lambda$ on the right-hand side and using $-\frac{1}{2} < \text{Im}\, \lambda <\frac{1}{2}$. The last inequality comes from the definition of $g$. 
    Thus by our choice of $A>C'$, we guarantee that this term decays faster than any polynomial.
    
    \noindent Turning to the contributions of $\Gamma_{1^\pm}$, consider the component $\Gamma_{1^+}=\{r(t) \times i[-\gamma (r(t)),0]\}$,
    \begin{align*}
        &\left\| \int_{-\gamma (r(t))}^{0} e^{-it(r(t)+iy)} \chi R(-r(t) -iy) \chi\, \langle (r(t)+iy)^2 \rangle^{-s} g dy \right\|_{\mathcal{D}^q} \\
        &\leq Ce^{C'r(t)} \int_{-\gamma (r(t))}^{0} e^{ty}\, |\langle (r(t)+iy)^2 \rangle|^{-s} dy\, \|g\|_{\mathcal{D}^p} \\
        &\leq Ce^{C'r(t)} \int_{-\gamma (r(t))}^{0} e^{ty} dy\, \|g\|_{\mathcal{D}^p}  \\
        &\leq \frac{1}{t} \left(1-e^{-t\gamma(r(t))}\right) Ce^{C'r(t)} \|f\|_{\mathcal{D}^{p+s}}.
    \end{align*}
    \noindent After simplification we see that this term decays as fast as $t^{\frac{C'}{A}-1}$, 
     so it  has a bound as claimed for 
    $u_r(t)$. The same argument shows the same estimate for $\Gamma_{1^-}$.

    We are left to consider the contribution of $R(\lambda +i0)$. We perform a contour deformation of the interval $(-r(t),r(t))$  to the contour $\Gamma_{1^-}\cup \Gamma_{2^-}\cup \Gamma_{\eta} \cup \Gamma_{2^+} \cup \Gamma_{1^+}$ of Figure \ref{fig:contour1}; for this term we avoid the negative imaginary axis as the integrand is not necessarily analytic there. This yields
    \begin{align*}
        \int_{-r(t)}^{r(t)} e^{-it \lambda} & \chi R(\lambda +i0) \chi f\, d\lambda =
        \int_{\Gamma_{2^\pm}} + \int_{\Gamma_\eta} + \int_{\Gamma_{1^\pm}}\\
    \end{align*}

    \noindent Over $\Gamma_{2^+}$,
    proceeding as in the case of $R(-\lambda +i0)$, using $|\langle \lambda^2 \rangle| \geq \frac{1}{2}$, we obtain
    \[
    \Big\|\int_{\Gamma_{2^+}}  e^{-it \lambda}  \chi R(\lambda +i0) \chi f\, d\lambda\Big\|_{\mathcal{D}^q} \le e^{-t \gamma(r(t))} C e^{C'r(t)} (r(t) - \eta)  \|g\|_{\mathcal{D}^p}.
    \]
    Taking $\eta$ to $0^+$, as in the estimate for $R(-\lambda+i0)$, the last inequality decays faster than any polynomial. The same argument works for the integral over $\Gamma_{2^-}$. \\

    \noindent The estimates for the integral over $\Gamma_{1^\pm}$ 
 follow exactly as our estimates for the corresponding term with $R(\lambda-i0)$.  Thus 
    far, each of the terms we have considered from \eqref{eq:lgsm} has contributed to the remainder,
    and all that remains is 
    the integral over $\Gamma_\eta$ in the limit $\eta \rightarrow 0^+$.
    That is,
    \[
    \chi \left(\frac{\sin(t\sqrt{P})} {\sqrt{P}}f - u_d(t)\right)=\frac{1}{2\pi} \lim_{\eta
    \to 0^+} \int_{\Gamma_\eta} e^{-it \lambda} \chi R(\lambda)\chi \,f\,d\lambda + O\left(\frac{1}{\log (t)^{2(s+p-q)+1}}\right) \|f\|_{\mathcal{D}^{p+s}}.
    \]

    \noindent Using Assumption 4
    we can write the integral over $\Gamma_\eta$ as the expression for $\chi u_z(t)$ in the statement of the Theorem.
\end{proof}

\section{Applications of the Main Expansion} 
\label{sec:thm1}

To prove Theorem \ref{t:intro}, we first check the hypotheses of Theorem \ref{t:main}. The high energy bounds on the resolvent in the Assumption (2) follows from Theorem 3 of \cite{burq02}, in the same way that Proposition 8.1 of \cite{burq02} does. Then assumptions (1),  the low energy part of (2), and (3) follow from Theorem 1 of \cite{cd2}.

 To facilitate the discussion in the proof of Theorem \ref{t:intro}, we introduce, for $l\in \mathbb{Z}$, the spaces
$$\mathcal{G}_l=\{ u\in \mathcal{D}_{\operatorname{loc}}: \; Pu=0\;\text{and} \; u(x)=O(|x|^{-l})\; \text{as $|x|\rightarrow \infty$}\}.$$
Notice that an element of $\mathcal{G}_{-1}\setminus \mathcal{G}_{-2}$ is a $p$-resonant state, and a nonzero element of $\mathcal{G}_{-2}$ is an eigenfunction.

\begin{proof}[Proof of Theorem \ref{t:intro}]
Part (i) follows directly from (1) of Theorem \ref{t:main}. For part (ii), recall from \cite[Section~3.4]{cdy} that, for some $M \in \{0, 1,2\}$, near $\lambda =0$ we have
\begin{equation}
\label{eq:soex}
    R(\lambda)=A_{-2,0}\lambda^{-2}+ \left(\sum_{m=1}^M \frac{U_{\omega_m}\otimes U_{\omega_m}}{\log (\lambda b_{-2,-1,m})}\right)\lambda^{-2} + A_{0,1} \log \lambda + \sum_{j=0}^\infty \sum_{k=-\infty}^{k_0(j)} A_{2j,k}\lambda^{2j} (\log \lambda)^k.
\end{equation}
Here $U_{\omega_m}\otimes U_{\omega_m} f= U_{\omega_m}\int f \overline{U}_{\omega_m}dx $, $U_{\omega_m}\in \mathcal{G}_{-1}\setminus \mathcal{G}_{-2}$ and 
$\arg b_{-2,-1,m}=-\pi/2$.
Moreover, $A_{-2,0}f= -\Pi_0f \in \mathcal{G}_{-2}, $ 
$A_{0,1}$ maps $L^2_c(\Real^2)$ to the span of $\mathcal{G}_{-2}\setminus \mathcal{G}_{-3}$ and $\mathcal{G}_{0}\setminus\mathcal{G}_{-1}$, $k_0(j)<\infty$, and $k_0(0)=0.$ The first term in the equation \eqref{eq:uz} follows from using the residue theorem on the term $A_{-2,0}\lambda^{-2}$ in the expansion \eqref{eq:soex}. For the second term in equation \eqref{eq:uz}, we set
\begin{equation}\label{eq:Jm}
\mathfrak{J}_m(t) =\lim_{\eta \rightarrow 0^+}\frac{1}{2\pi}\int_{\Gamma_{\eta}}e^{-it\lambda}\lambda^{-2}(\log (\lambda b_{-2,-1,m}))^{-1} d\lambda, \quad 0<m\leq M.
\end{equation}
By \cite[Lemma~2.3]{cdy}, we have $\mathfrak{J}_m(t)= t( (\log t)^{-1}+O((\log t)^{-2}))$ as $t\rightarrow \infty$. Note that although the contour here is slightly different from the contour in \cite[Lemma~2.3]{cdy}, the difference will only cause an exponentially small error. By \cite[Proposition 5.12]{cd2}, $\mathfrak{J}_1 = \mathfrak{J}_2$ whenever $V$ is radial. The  contribution of the rest of the terms in the expansion \eqref{eq:soex} decay faster than the term $u_r$ in Theorem \ref{t:intro} by \cite[Lemma~2.3]{cdy}. 

To prove equation \eqref{e:ur} in part (iii) using part (3) of Theorem \ref{t:main}, note that if $\chi \in C_c^\infty(\Omega)$, then, by interior ellipticity of $P$, 
\[
\|\chi u\|_{\mathcal D^q} \lesssim \|\chi u\|_{H^{q/2}(\Omega)}  \lesssim \|\chi u\|_{\mathcal D^q} ,
\]
for all $u$, where the implicit constants depend on $\chi$, and similarly, by Sobolev embedding,
\[
\|\chi u\|_{H^k(\Omega)} \lesssim \|\chi u\|_{C^k(\Omega)}  \lesssim \|\chi u\|_{H^s(\Omega)} ,
\]
as long as $s>k+1$. 
\end{proof}

\section{Examples with $p$-resonance} 
\label{sec:p-res}

This section constructs five different classes of operators which have $p$-resonances.  Moreover, each fits into the general framework of \cite{cd2}, so that 
the existence of low-energy resolvent expansions is essentially immediate.  For the first three, we make no effort to construct the most general operators possible within the class.   In each of these three, our technique is first to construct what will be our $p$-resonant state, and then use this to finish construction of the operator.  The first three examples satisfy the hypotheses of Theorem \ref{t:intro} (with appropriate assumptions on $M$ for the second), and the last two satisfy the hypotheses of Theorem \ref{t:main}.  
\subsection{Variable wave speed Schr\"{o}dinger operator}\label{ss:vws}
We begin with the operator $-c(x)^2 \Delta$, where the (possibly variable) wave speed  $c\in C^\infty(\Real^2) $ satisfies $0<c_m\leq c(x)\leq c_M<\infty$.   
Assuming that there is a $c_0>0$ so that $c(x)=c_0$ outside of a compact set allows the immediate application of the results of \cite{cd2}, but is stronger than necessary for the construction of an 
operator with a $p$-resonance.
We construct a compactly supported smooth potential $V$ so that $-c^2\Delta +V$ has a p-resonance at $0$. 

Let $a_0>0$ be fixed, and let  $\chi \in C_c^\infty(\Real^2)$ so that $\chi(x)$ is equal to $1$ in a neighborhood of the origin and $0\leq \chi (x)\leq 1$ for all $x$.  We shall 
make one  additional requirement on $\chi$ below.
Write $x=(x_1,x_2)$, and set 
$$u_p(x)= (1-\chi(x))\frac{x_1}{|x|^2}+ \chi(x) a_0x_1.$$
Note
\begin{equation} \label{eq:deltaup}
-\Delta u_p=  (\Delta \chi)\left( \frac{x_1}{|x|^2}- a_0x_1\right) + 2 (\partial_{x_1}\chi)\left( \frac{x_2^2-x_1^2}{|x|^4} - a_0\right)+2 (\partial_{x_2}\chi) \left( \frac{2x_1 x_2}{|x|^4}\right).
\end{equation}
We want to ensure that $ \Delta u_p$ vanishes wherever $u_p$ vanishes.  Notice that by our choices of $a_0$ and $\chi$, $u_p(x)=0$ if and only if $x_1=0$.  Moreover,
$$(\partial_{x_1}  u_p)(0,x_2)= (1-\chi (0,x_2))\frac{1}{x_2^2}+\chi(0,x_2)a_0>0$$
so that $u_p$ vanishes only to first order on $\{x_1=0\}$.
To ensure that 
$\Delta u_p$ vanishes at $x_1=0$, we need only to require  $(\partial_{x_1}\chi)(0,x_2)=0$.  This is certainly possible, for example, by taking $\chi$ to depend only on $x_2$ in a small neighborhood of $\{x_1=0\}$.  Then set 
$$V= c^2 \frac{\Delta u_p}{u_p}$$
and note $(-c^2\Delta +V)u_p=0$.
By our  construction of $u_p$, $V$ is a smooth, compactly supported real-valued function.  Notice that we do not require symmetry conditions on $c$ or $\chi$.

\subsection{Schr\"{o}dinger operators on manifolds}
Let $(M,g)$ be a two-dimensional smooth Riemannian manifold so that $M$ can be written as the union of a compact set and a set isomorphic to the complement of a ball in $\Real^2$ in a sense
described below.  We shall demonstrate how to construct a potential $V\in C_c^\infty(M)$ so that $-\Delta_g+V$ has a $p$-resonance.  Here $-\Delta_g\geq 0$ is the Laplacian on $(M,g)$.

Denote by $B(x_0,\rho)$ the open ball of radius $\rho$ and center $x_0$ in $\Real^2$.  Suppose $M\simeq K\sqcup \Real^2\setminus \overline{B}((4\rho,0), \rho)$ for some compact set $K\subset M$ and some $\rho>0$.  Moreover, suppose
$g\upharpoonright_{M\setminus K}=g_0$, where $g_0$ is the usual Euclidean metric.  We  identify points in $\Real^2\setminus \overline{B}((4\rho,0), \rho) $ with points in $M\setminus K$.
 In a similar manner as in Section \ref{ss:vws}, we construct a function $u_p$ which shall be our $p$-resonant state.  Let $\chi_K\in C_c^\infty(M)$, $0\leq \chi_K\leq 1$, be $1$ in a neighborhood of $K$ and $0$ on 
 $\{x=(x_1,x_2)\in \Real^2: x_1<2\rho\}$.  Let $\chi_0\in C_c^\infty(M)$, $0\leq \chi_0\leq 1$  be $1$ on $B(0,\rho/2)$ and $0$ outside $B(0,\rho)$.   Moreover, require that $\partial_{x_1}\chi_0\upharpoonright_{x_1=0}=0$.
  Now let $a_0,\; a_K>0$ be constants, and set 
 $$u_p= (1-\chi_0-\chi_K)\frac{x_1}{|x|^2}+ a_0\chi_0 x_1 + a_K\chi_K.$$
 Then, as in Section \ref{ss:vws}, our choices of $\chi_0$, $\chi_K$, $a_0$ and $a_K$ ensure that 
 $$\{x\in M: u_p(x)=0\} \subset \{x=(x_1,x_2)\in M\setminus K:  x_1=0\}\subset M.$$  Moreover, again as in Section \ref{ss:vws}, the condition that 
 $\partial_{x_1}\chi_0\upharpoonright_{x_1}=0$ ensures that $\Delta_g u_p/u_p$ is smooth.  Now define $V:=(\Delta_g u_p)/u_p$, and $(-\Delta_g+V)u_p=0$.

\subsection{Robin boundary conditions}  For a large class of obstacles $\mco$ in $\Real^2$ we construct a boundary condition so that $-\Delta$ on $\Real^2 \setminus \overline{\mco}$ has a $p$-resonance at $0$.

Let $\mco \subset \Real^2$ be an open set with smooth boundary so that  $0\in \mco$, $\overline{\mco}$ is compact, and $\Real^2 \setminus \overline{\mco}$ is connected.   
We shall put an additional requirement on 
$\mco$ below.  Set $u_p(x)= x_1/|x|^2$, so that $u_p$ is smooth on $\Real^2 \setminus \overline{\mco}$ and $-\Delta u_p=0$ there.  Let $\partial_\nu$ denote differentiation with respect to
the outward pointing (to $\Real^2 \setminus \overline{\mco}$) unit normal vector on $\partial \mco$.  We want $(\partial_\nu u_p)/u_p\in C^\infty(\partial \mco)$.  This holds if $\partial \mco$ has a horizontal tangent at its intersections with  the $x_2$ axis, as in that case $\partial_\nu u_p$ vanishes to at least first order on the set on which  $u\upharpoonright_{\partial \mco}$  vanishes. With this assumption, 
$\sigma:= (\partial_\nu u_p)/u_p\in C^\infty(\partial \mco)$ and the operator $-\Delta $ on $\Real^2 \setminus \overline{\mco}$ with domain $\{u\in H^2(\Real^2 \setminus \overline{\mco}):\; \partial_\nu u= \sigma u\}$ has a $p$-resonance at $0$.

A simple example is the case $\mco$ is a ball of radius $\rho$, for which $\sigma=1/\rho$ provides a function for which the Robin problem has a $p$-resonance.  In fact, this operator has two linearly
independent $p$-resonant states, $x_1/|x|^2$ and $x_2/|x|^2$.

\subsection{Schr\"odinger operators with round wells}
We show that compactly supported radial potentials
\begin{equation}
    \label{eq:potential}
     V(r) = -a^2 \mathbf{1}_{\{r < R\}}, \quad a>0
\end{equation}
given by the characteristic function of a disc can admit a $p$-resonant state; see also \cite{cdg}. 

For a solution with angular momentum $ m \in \mathbb{Z} $, we separate variables as $ u(r, \theta) = f(r)e^{im\theta} $. The equation $-\Delta u + Vu =0$ becomes:
\begin{equation}
\label{eq:radial}
    -f''(r) - \frac{1}{r}f'(r) + \left(\frac{m^2}{r^2} + V(r)\right)f(r) = 0.
\end{equation}
For $ m \neq 0 $, the effective potential outside the disc $D_R=\{r<R\}$ is given by $ \frac{m^2}{r^2} $. Equation \eqref{eq:radial} then becomes:
\[
f''(r) + \frac{1}{r}f'(r) - \frac{m^2}{r^2}f(r) = 0.
\]
with general solutions: 
\[
f(r) = A r^{|m|} + B r^{-|m|}.
\]
For $p$-resonant states, we require $ A = 0 $ and $ |m| = 1 $. This gives
\[
f(r) = \frac{B}{r} \quad \text{for } r \ge R.
\]
For $ r < R $ the equation becomes:
\[
f''(r) + \frac{1}{r}f'(r) - \frac{1}{r^2}f(r) + a^2 f(r) = 0.
\]
This is the Bessel equation with $L^2$ solutions $ f(r) = C J_1(ar) $, where $ J_1 $ is the Bessel function of the first kind $J_{\nu}$ with the order $\nu=1$. 

Now we consider the matching conditions at $ r = R $ which are
\begin{equation}
    f(R-)=f(R+), \quad f'(R-) = f'(R+).
\end{equation}
This yields
   \[
   C J_1(aR) = \frac{B}{R}, \quad C a J_1'(aR) = -\frac{B}{R^2}.
   \]
   Using the recurrence relation $ J_1'(z) = J_0(z) - \frac{J_1(z)}{z} $ from \cite[10.6.2]{DLMF} these conditions reduce to
   \[
   J_0(aR) = 0.
   \]
The condition is satisfied when $ a R = j_{0,n} $, where $ j_{0,n} $ are zeros of the of the Bessel function $ J_0 $. Thus, for 
   \[
   a = \frac{j_{0,n}}{R}, \quad n\in \mathbb{N}
   \]
in the potential \eqref{eq:potential}, the Schr\"odinger operator has $p$-resonant states given by
\begin{equation*}
    u(r,\theta) =
    \begin{cases} 
C J_1(ar) e^{\pm i\theta}, & r < R, \\
 r^{-1} e^{\pm i\theta}, & r \geq R
\end{cases}
\quad \text{ with } C = 1/(R J_1(aR)).
\end{equation*}

\begin{rem}
    When $a^2<0$, there are no $p$-resonances. This is because \eqref{eq:radial} is now a modified Bessel equation within the disc, with solutions given by modified Bessel functions $I_1$. The same argument reduces to the zeros of modified Bessel functions $I_0$, which has no real zeroes \cite[\S 10.42]{DLMF}. 
\end{rem}
Assumption (2) follows from the proof of \cite[Theorem~2.10]{dz} on the high-frequency behavior of the resolvent and the free resolvent bound in \cite[Section~2]{gs}. Assumption (3) follows from \cite{cd2}. Hence, we can apply Theorem \ref{t:main} to Schr\"odinger operators with the potential in \eqref{eq:potential} to obtain the corresponding wave decay rate in the presence of $p$-resonance.

\subsection{Schr\"odinger operators with delta-potential rings}
We show that Schr\"odinger operators with delta potential
\begin{equation}
\label{eq:delta-potent}
    V(r) = a \delta(r - R)
\end{equation}
supported in a circle admit $ p $-resonance states for suitable $a$. Using separation of variables $ u(r, \theta) = f(r)e^{im\theta} $, the equation becomes:
\begin{equation}\label{eq:dpot}
-f'' - \frac{1}{r}f' + \frac{m^2}{r^2}f + a \delta(r - R)f = 0.
\end{equation}
Since a $p$-resonant state must
in this case be a linear combination
of solutions of \eqref{eq:dpot} with 
$|m|=1$, 
using the local $L^2$-integrability, 
in order to have a $p$-resonant 
state $f$ must satisfy
\[
f(r) = 
\begin{cases} 
A r, & r < R, \\
B r^{-1}, & r > R.
\end{cases}
\]
It remains to consider matching conditions at $ r = R $. In particular, the continuity of $ f(r) $ at $r=R$ forces
\begin{equation}
\label{eq:cont}
     A R = B R^{-1}. 
\end{equation}
whereas the jump in derivatives due to the delta potential yields
\begin{equation}
    \label{eq:jump}
      f'(R+) - f'(R-) = a f(R).
\end{equation}
Note that by \eqref{eq:cont}, we have 
  \[
  f'(R^-) = A, \quad f'(R^+) = -B R^{-2} = -A.
  \]
Combining with the jump condition \eqref{eq:jump}, we obtain $p$-resonant states
  \[
u(r,\theta) = 
\begin{cases} 
r e^{\pm i\theta}, & r < R, \\
R^2 r^{-1} e^{\pm i\theta}, & r > R,
\end{cases}
\quad \text{ if } a = -\frac{2}{R}.
\]
Assumption (2) then follows from the resolvent bound in \cite[Lemma~7.1]{gs}; Assumption (3) follows from \cite{cd2}. This allows us to apply Theorem \ref{t:main} to Schr\"odinger operators with the potential in \eqref{eq:delta-potent}.

\end{document}